\documentclass[a4paper]{amsart}
\usepackage{amsmath, amscd}
\usepackage{amssymb, color, hyperref}

\usepackage{amsmath}
\usepackage{amssymb, color}
\usepackage{url}
\usepackage{hyperref}

\usepackage[fontsize=11pt]{scrextend}

\theoremstyle{plain}
\newtheorem{theorem}{\bf Theorem}[section]
\newtheorem{proposition}[theorem]{\bf Proposition}
\newtheorem{lemma}[theorem]{\bf Lemma}

\theoremstyle{definition}
\newtheorem{example}[theorem]{\bf Example}

\newtheorem{definition}[theorem]{\bf Definition}

\newcommand{\N}{\mathbb N}
\newcommand{\Z}{\mathbb Z}
\newcommand{\R}{\mathbb R}
\newcommand{\Q}{\mathbb Q}

\newcommand{\DP}{\negthinspace : \negthinspace}
\renewcommand{\t}{\, | \,}
\newcommand{\LK}{\,[\![}
\newcommand{\RK}{]\!]}
\newcommand{\BF}{\text{\rm BF}}

\DeclareMathOperator{\Int}{Int} 
   \DeclareMathOperator{\id}{id}

\newcommand{\red}{{\text{\rm red}}}

\renewcommand{\time}{\negthinspace \times \negthinspace}

\numberwithin{equation}{section}

\makeatletter
\newcommand\zeu@Scale{1.25}
\makeatother

\begin{document}

\author{Qinghai Zhong}

\address{University of Graz, NAWI Graz \\
Institute for Mathematics and Scientific Computing \\
Heinrichstra{\ss}e 36\\
8010 Graz, Austria}

\email{qinghai.zhong@uni-graz.at}
\urladdr{http://qinghai-zhong.weebly.com/}



\thanks{This work was supported by the Austrian Science Fund FWF, Project Number P28864-N35}

\begin{abstract}
Let $H$ be a commutative and cancellative monoid. The elasticity  $\rho(a)$ of a non-unit $a \in H$ is the supremum  of  $m/n$ over all $m, n$ for which there are factorizations of the form $a=u_1 \cdot \ldots \cdot u_m=v_1 \cdot \ldots \cdot v_{n}$, where all $u_i$ and $v_j$ are irreducibles. The elasticity $\rho (H)$ of $H$ is the supremum over all $\rho (a)$. We establish a characterization, valid for finitely generated monoids, when every rational number $q$ with $1< q < \rho (H)$ can be realized as the elasticity of some element $a \in H$. Furthermore, we derive results of a similar flavor for locally finitely generated monoids (they include all Krull domains and orders in Dedekind domains satisfying certain algebraic finiteness conditions) and for weakly Krull domains.
\end{abstract}

\title{On elasticities of locally finitely generated monoids}
\maketitle

\medskip
\section{Introduction} \label{1}
\medskip
In this paper, a {\it monoid} means a commutative cancellative semigroup with identity element and the monoids we mainly have in mind are multiplicative monoids of nonzero elements of domains. A monoid is said to be locally finitely generated if for every given element $a$ there are only finitely many irreducibles (up to associates) which divide some power of $a$. Krull domains are locally finitely generated and more examples are given in Section \ref{2} (see Examples \ref{2.1}).

Let $H$ be a monoid and $a \in H$. If $a$ has a factorization into irreducibles, say $a = u_1\cdot \ldots\cdot u_k$, then $k$ is called a factorization length and the set $\mathsf L (a) \subset \N$ of all possible factorization lengths is called the set of lengths of $a$. For convenience we set $\mathsf L (a) = \{0\}$ if $a$ is a unit. The monoid $H$ is said to be a BF-monoid if every non-unit has a factorization into irreducibles and all sets of lengths are finite. It is well-known that $v$-noetherian monoids  are BF-monoids.

Suppose that $H$ is a BF-monoid. The system $\mathcal L (H) = \{\mathsf L (a) \mid a \in H \}$ of sets of lengths, and all parameters controlling $\mathcal L (H)$, are a well-studied means to describe the non-uniqueness of factorizations of BF-monoids. Besides distances of factorizations, elasticities belong to the main parameters. For a finite set $L \subset \N$, $\rho (L) = \max L/\min L $ denotes the elasticity of $L$ and, for an element $a \in H$, the elasticity $\rho (a)$ of $a$ is the elasticity of its set of lengths. The elasticity $\rho (H)$ of $H$ is the supremum of $\rho (L)$ over all $L \in \mathcal L (H)$.

 Since the late 1980s various aspects of elasticities have found wide attention in the literature. We refer to a survey by David F. Anderson \cite{An97a} for work till 2000 and to  \cite{Ch-St10a, Ge-Gr-Yu15, Ba-Co16a, Ka16b, Ba-Ne-Pe17a} for a sample of papers in the last years. To mention some results explicitly, we recall that
for every $r \in \R_{\ge 1} \cup \{\infty\}$, there is a Dedekind domain $R$ with torsion class group such that $\rho(R)=r$ (\cite{An-An92}).
 A characterization of when the elasticity of finitely generated domains is finite is given in \cite{Ka05a}.   The elasticity of  C-monoids (they include wide classes of Mori domains with nontrivial conductor) is rational or infinite (\cite{Zh19b}).

In \cite{Ch-Ho-Mo06}, Chapman et al. initiated the study of the set $\{\rho (L) \mid L \in \mathcal L (H) \} \subset \Q_{\ge 1}$ of all elasticities.
We say that $H$ is {\it fully elastic} if for every rational number $q$ with $1 < q < \rho (H)$ there is an $L \in \mathcal L (H)$ such that $\rho (L)=q$.
Monoids having accepted elasticity  and  having a prime element are fully elastic (\cite{B-C-C-K-W06}), and it was shown only recently that all  transfer Krull monoids (they include Krull domains and  wide classes of non-commutative Dedekind domains) are fully elastic (\cite[Theorem 3.1]{Ge-Zh19a}).
On the other hand, strongly primary monoids (including one-dimensional local Mori domains and numerical monoids) are not fully elastic (\cite[Theorem 5.5]{Ge-Sc-Zh17b}).  Arithmetic congruence monoids which are not fully elastic can be found in the survey \cite{Ba-Ch14a}.

Anderson and Pruis \cite{An-Pr91} studied, for every $a \in H$, the quantities
\[
\rho_* (a) = \lim_{n \to \infty} \frac{\min \mathsf L (a^n)}{n} \quad \text{and} \quad \rho^* (a) = \lim_{n \to \infty}\frac{\max \mathsf L (a^n)}{n}
\]
(see also \cite{An-Ju17a}) and they conjectured these invariants are rational for Krull domains and for noetherian domains. This was confirmed in \cite{Ge-HK92b} for Krull domains and for various classes  of noetherian domains but is still open in general. In \cite{B-C-H-M06}, Baginski et al. introduced the concept of asymptotic elasticities.
For $a \in H\setminus H^{\times}$,
\[
\overline{\rho}(a)=\lim_{n\rightarrow\infty}\rho(a^n) = \frac{\rho^* (a)}{\rho_* (a)}
\]
is the  {\it asymptotic elasticity} of $a$, and
\[
\overline{R}(H)=\{\overline{\rho}(a)\mid a \in H\setminus H^{\times} \}\subset \R_{\ge 1}\cup\{\infty\}
\]
denotes the {\it  set of asymptotic elasticities} of $H$.
We say that $H$ is {\it asymptotic fully elastic} if for every rational number $q$ with $\inf \overline{R}(H) < q < \rho (H)$ there is an $a\in H$ such that $\overline{\rho} (a)=q$ (note that $\sup \overline{R}(H)=\rho(H)$). Now we can formulate our main results.

\smallskip
\begin{theorem} \label{theorem1}
Let $H$ be a locally finitely generated monoid. Then $H$ is asymptotic fully elastic and if
 $\inf \overline{R}(H)=1$, then $H$ is fully elastic.
\end{theorem}

Every Krull monoid $H$ is locally finitely generated with $\inf \overline{R}(H)=1$ (\cite[Proposition 2.7.8.3]{Ge-HK06a} and \cite[Lemma 5.4]{Ge-Sc-Zh17b}).
However, there are   locally finitely generated monoids $H$ with $\inf \overline{R}(H)>1$ that are fully elastic (Example \ref{ex}).

\begin{theorem}\label{theorem2}
Let $H$ be a monoid such that $H_{\red}$ is  finitely generated and let $r=\inf  \overline{R}(H)$. Then
$$\{q\in \Q\mid r\le q\le \rho(H)\}\subset \{\rho(L)\mid L\in \mathcal L(H)\}$$
and $r$ is the only possible limit point of the set $\{\rho(L)\mid L\in \mathcal L(H)\text{ and }1\le \rho(L)< r\}$.
Moreover, $H$ is fully elastic if and only if $r=1$.
\end{theorem}

In Section \ref{2} we provide the required background. The proofs of Theorems \ref{theorem1} and \ref{theorem2} are  given in Section \ref{3}, and then in  Section \ref{4} we apply our results to $v$-noetherian weakly Krull monoids.

\medskip
\section{Background on monoids and their arithmetic} \label{2}
\medskip

Our notation and terminology are consistent with \cite{Ge-HK06a}.
 Let $\mathbb N$ denote the set of positive integers and $\N_0=\N\cup\{0\}$. For $k \in \N$, we denote by $\N_{\ge k}$ the set of all integers greater than or equal to $k$ and
 for  $a, b \in \mathbb Q$, we denote
by $[a, b ] = \{ x \in \mathbb Z \mid a \le x \le b\}$ the discrete, finite interval between $a$ and $b$. For a finite subset $L\subset \N$, we set $\rho(L)=\max L/\min L$ and $\rho (\{0\})=1$.  For  subsets $A,B\subset \Z$,  $A+B=\{a+b\mid a\in A, b\in B\}$ denotes their sumset.

\smallskip
\noindent
{\bf Monoids.} Throughout this paper, a {\it monoid} means a commutative cancellative semigroup with identity element, and we use multiplicative notation. If $R$ is a domain, then its multiplicative semigroup $R^{\bullet} = R \setminus \{0\}$ of nonzero elements is a monoid.

Let $H$ be a monoid with identity element $1=1_H\in H$. We denote by $H^{\times}$  the unit group of $H$,  by $H_{\red}=H/H^{\times}=\{aH^{\times}\mid a\in H\}$  the associated reduced monoid, and by $\mathsf q(H)$  the quotient group of $H$. If $H^{\times}=\{1\}$, we say that $H$ is reduced. Two elements $a, b \in H$ are said to be associated if $aH^{\times} = b H^{\times}$. 
A submonoid $S \subset H$ is said to be 
\begin{itemize}
\item saturated if $a,b\in S$ and $a\t_H\, b$ implies that $a\t_S\, b$ ( equivalently, $S=\mathsf q(S)\cap H$ ).

\item divisor-closed if $a \in S$, $b \in H$, and $b \mid a$ implies that $b \in S$.
\end{itemize}
 If $a \in H$, then
\[
\LK a \RK = \{ b \in H \mid b \mid a^n \ \text{for some $n \in \N$} \} \subset H
\]
is the smallest divisor-closed submonoid of $H$ containing $a$.
   An element $u \in H$ is said to be irreducible (or an atom) if $u \notin H^{\times}$ and any equation of the form $u=ab$, with $a, b \in H$, implies that $a \in H^{\times}$ or $b \in H^{\times}$. Let  $\mathcal A (H)$ denote the set of atoms. 
    For a set $\mathcal P$, we denote by $\mathcal F (\mathcal P)$ the \ {\it free
   abelian monoid} \ with basis $\mathcal P$. Then every $a \in \mathcal F (\mathcal P)$ has a unique representation of the form
   \[
   a = \prod_{p \in \mathcal P} p^{\mathsf v_p(a) } \quad \text{with} \quad
   \mathsf v_p(a) \in \N_0 \ \text{ and } \ \mathsf v_p(a) = 0 \ \text{
   for almost all } \ p \in \mathcal P \,,
   \]
   and we  call $|a|_{\mathcal F (\mathcal P)} = |a|= \sum_{p \in \mathcal P}\mathsf v_p(a)$ the \emph{length} of $a$.

   The monoid $H$ is said to be
\begin{itemize}
\item {\it atomic} if every non-unit is a finite product of atoms.

\item {\it factorial} if it is atomic and every atom is prime.

\item {\it finitely generated} if it has a finite generating set (equivalently, $H_{\red}$ and $H^{\times}$ are both finitely generated).

\item {\it locally finitely generated} if for every $a \in H$ the monoid $\LK a \RK_{\red  }$ is finitely generated (equivalently, there are only finitely many atoms (up to associates) dividing some power of $a$).
\end{itemize}
We gather some examples of locally finitely generated monoids.

\smallskip
\begin{example} \label{2.1}~

1. Clearly, every monoid $H$ having (up to associates) only finitely many atoms that are not prime is locally finitely generated. Atomic domains, with almost all atoms being prime, are called generalized Cohen-Kaplansky domains and they were introduced in \cite{An-An-Za92b}. The monoid of integral invertible ideals of a domain is finitely generated if and only if the domain is a Cohen-Kaplansky domain (\cite[Theorem 4.3]{An-Mo92}).

\smallskip
2. Saturated submonoids (whence, in particular, divisor-closed submonoids) and coproducts of locally finitely generated monoids are locally finitely generated (\cite[Proposition 2.7.8]{Ge-HK06a}). Thus, Krull monoids are locally finitely generated.

\smallskip
3. Let $R$ be a factorial domain. Then the ring of integer-valued polynomials is (in general) not a Krull domain. But for every $f \in \Int (R)$, the submonoid $\LK f \RK \subset \Int (R)$ is Krull (\cite[Theorem 5.2]{Re14a}) whence $\Int (R)$ is locally finitely generated.

\smallskip
4. Let $R$ be an order in a Dedekind domain, say $R \subset \overline R$, where $\overline R$ is the integral closure of $R$ and $\overline R$ is a Dedekind domain. If the class group $\mathcal C (\overline R)$ and  the residue class ring $\overline R/(R \DP \overline R)$ is finite, and $R$ has finite elasticity, then $R$ is locally finitely generated (see \cite[Corollary 3.7.2]{Ge-HK06a} and \cite[Theorem 3.7.1]{Ge-HK06a} for a more general result in the setting of weakly Krull domains).
\end{example}

\smallskip
\noindent
{\bf Arithmetic of monoids.}
If $a \in H \setminus H^{\times}$ and $a=u_1 \cdot \ldots \cdot u_k$, where $k \in \N$ and $u_1, \ldots, u_k \in \mathcal A (H)$, then $k$ is a factorization length of $a$, and
\[
\mathsf L (a) = \{k \mid k \ \text{is a factorization length of} \ a \} \subset \N
\]
denotes the {\it set of lengths} of $a$. It is convenient to set $\mathsf L (a) = \{0\}$ for all $a \in H^{\times}$. The family
\[
\mathcal L (H) = \{\mathsf L (a) \mid a \in H \}
\]
is called the {\it system of sets of lengths} of $H$. The monoid $H$ is said to be
\begin{itemize}
\item {\it half-factorial} if it is atomic and $|L|=1$ for every $L \in \mathcal L (H)$.

\item {\it a \BF-monoid} if it is atomic and $L$ is finite for every  $L \in \mathcal L (H)$.

\end{itemize}
Clearly, every factorial monoid is half-factorial (but not conversely) and every $v$-noetherian monoid is a \BF-monoid (\cite[Theorem 2.2.9]{Ge-HK06a}).
Let $H$ be a \BF-monoid. The {\it elasticity} $\rho (a)$ of an element $a \in H$ is defined as the elasticity of its set of lengths whence
\[
\rho (a) = \rho ( \mathsf L (a) ) = \frac{\max \mathsf L (a)}{\min \mathsf L (a)}
\]
and the supremum
\[
\rho (H) = \sup \{\rho (L) \mid L \in \mathcal L (H) \} \in \R_{\ge 1} \cup \{\infty\}
\]
denotes the {\it elasticity} of $H$. The monoid $H$
\begin{itemize}
\item has {\it  accepted elasticity} if there is some $L \in \mathcal L (H)$ such that $\rho (L)= \rho (H)$.

\item is {\it fully elastic} if for every $q \in \Q$ with $1 < q < \rho (H)$ there is some $L \in \mathcal L (H)$ such that $\rho (L)=q$.
\end{itemize}
The {\it asymptotic elasticity $\overline{\rho}(a)$} of an element $a \in H\setminus H^{\times}$ is defined as
\[
\overline{\rho}(a)=\lim_{n\rightarrow\infty}\rho(a^n) \qquad \text{(note that  the limit exists by \cite[Theorem 3.8.1]{Ge-HK06a})} \,,
\]
and
\[
\overline{R}(H)=\{\overline{\rho}(a)\mid a\in H\setminus H^{\times} \} \subset \R_{\ge 1} \cup \{\infty\}
\]
denotes the {\it set of asymptotic elasticities}. We conclude this section with a technical lemma.

\medskip
\begin{lemma}\label{3.2}
Let $n\in \N$ and $a_1,\ldots, a_n\in \N$ be positive integers. If there exist $x_1,\ldots, x_n\in \N_0$ and $t\ge 2$ such that $a_1x_1+\ldots+a_nx_n=ta_1\cdot\ldots\cdot a_n$, then there exist $x_i'\in [0,x_i]$, for all $i\in [1,n]$, such that
\[
a_1x_1'+\ldots+a_nx_n'=a_1\cdot\ldots\cdot a_n\,.
\]
In particular, there exist $x_i^{(j)}\in [0,x_i]$, for all $i\in[1,n]$ and $j\in[1,t]$, such that  $\sum_{j\in [1,t]}x_i^{(j)}=x_i$ for every $i\in [1,n]$ and $\sum_{i\in [1,n]}a_ix_i^{(j)}=a_1\cdot\ldots\cdot a_n$ for every $j\in [1,t]$.
\end{lemma}

\begin{proof}
The assertion is obvious for $n=1$. If $n=2$, then $a_1x_1\ge a_1a_2$ or $a_2x_2\ge a_1a_2$ whence the assertion follows immediately.
Suppose that $n \ge 3$ and  distinguish two cases.

\medskip
\noindent{\bf Case 1. } $\min \{a_1,\ldots, a_n\}\ge 2$.

After renumbering if necessary we
assume that  $a_1x_1\ge \frac{ta_1\cdot\ldots\cdot a_n}{n}$. Then
 \begin{equation}
x_1\ge \frac{2a_2\cdot\ldots\cdot a_n}{n}\ge \frac{2^{n-1}\min\{a_2,\ldots, a_n\}}{n}\ge \min\{a_2,\ldots, a_n\}\,.
 \end{equation}
For each  $i\in [2,n]$, we set $x_i=y_ia_1+r_i$ with $r_i\in[0, a_1-1]$. Then
\[
a_2r_2+\ldots+a_nr_n\le (a_1-1)(a_2+\ldots +a_n)\le a_1\cdot\ldots\cdot a_n\,.
\]
Therefore $a_1(x_1+y_2a_2+\ldots +y_na_n)\ge ta_1\cdot\ldots\cdot a_n-a_1\cdot\ldots\cdot a_n\ge a_1\cdot\ldots\cdot a_n$ which implies that
\[
x_1+y_2a_2+\ldots + y_na_n\ge a_2\cdot\ldots\cdot a_n\,.
\]
If $y_2a_2+\ldots +y_na_n\le a_2\cdot\ldots\cdot a_n$, then there exists $x_1'\in [0,x_1]$ such that $x_1'+y_2a_2+\ldots +y_na_n=a_2\cdot\ldots\cdot a_n$ and hence $a_1x_1'+a_2y_2a_1+\ldots +a_ny_na_1=a_1\cdot\ldots\cdot a_n$. If $y_2a_2+\ldots+ y_na_n>a_2\cdot\ldots\cdot a_n$, then we can choose $y_i'\in [0,y_i]$ for each $i\in [2,n]$ such that $$0\le a_2\ldots a_n-(y_2'a_2+\ldots +y_n'a_n)\le \min\{a_2,\ldots, a_n\}\le x_1\,.$$
Thus there exists $x_1'\in [0,x_1]$ such that $x_1'+y_2'a_2+\ldots y_n'a_n=a_2\cdot\ldots\cdot a_n$ and hence $a_1x_1'+a_2y_2'a_1+\ldots+ a_ny_n'a_1=a_1\cdot\ldots\cdot a_n$.

\medskip
\noindent{\bf Case 2. } $\min\{a_1,\ldots,a_n\}=1$.

After renumbering if necessary we
assume that  there exists $\tau\in [1,n]$ such that $a_{\tau}=a_{\tau+1}=\ldots=a_n=1$ and $a_i\ge 2$ for all $i\in [1,\tau-1]$. If $x_{\tau}+\ldots+x_n\ge a_1\cdot\ldots\cdot a_n$ or $\tau\le 2$, then the assertion follows immediately. Suppose $\tau=3$. Then $a_1x_1+a_2x_2+x_3+\ldots+x_n=ta_1a_2\ge 2a_1a_2$. If $a_1x_1\ge a_1a_2$ or $a_2a_3\ge a_2x_3$, then we are done. Otherwise $a_1x_1<a_1a_2$ and $a_1x_1+x_3+\ldots+x_n\ge a_1a_2$. We choose $x_i'\in [0,x_i]$ for all $i\in [3, n]$ such that  $\sum_{i=3}^nx_i'=a_1a_2-a_1x_1$. Then $a_1x_2+\sum_{i=3}^nx_i'=a_1a_2=a_1\cdot\ldots\cdot a_n$.

Now we assume that $\tau\ge 3$ and $x_{\tau}+\ldots+x_n<a_1\cdot\ldots\cdot a_n$

\smallskip
\noindent{\bf\  Case 2.1. } $\frac{ta_1\cdot\ldots\cdot a_n}{\tau}\le x_{\tau}+\ldots+x_n< a_1\cdot\ldots\cdot a_n$.

 Thus
\begin{align*}
&x_{\tau}+\ldots+x_n\ge \frac{2a_1\ldots a_{\tau}}{\tau}\ge \frac{2^{\tau-1}\min\{a_1,\ldots, a_{\tau-1}\}}{\tau}\ge \min\{a_1,\ldots, a_{\tau-1}\}\,.\\
 \text{and }\quad &a_1x_1+\ldots +a_{\tau-1}x_{\tau-1}>a_1\cdot\ldots\cdot a_n\,.
 \end{align*}

 We choose $x_i'\in [0, x_i]$ for each $i\in [1, \tau-1]$ such that
 \[
 0\le a_1\cdot\ldots\cdot a_n-(a_1x_1'+\ldots+ a_{\tau-1}x_{\tau-1}')\le \min\{a_1,\ldots, a_{\tau-1}\}\le x_{\tau}+\ldots+x_n\,.
 \]
 Therefore we choose $x_i'\in [0, x_i]$ for each $i\in [\tau+1,n]$ such that
 $x_{\tau}'+\ldots +x_n'=a_1\cdot\ldots\cdot a_n-(a_1x_1'+\ldots+ a_{\tau-1}x_{\tau-1}')$.
 It follows that $a_1x_1'+\ldots +a_nx_n'=a_1\cdot\ldots\cdot a_n$.

 \smallskip
 \noindent{\bf \ Case 2.2. } $\frac{ta_1\cdot\ldots\cdot a_n}{\tau}> x_{\tau}+\ldots+x_n$.

  Then there must exist $j\in [1,\tau-1]$ such that $a_jx_j\ge \frac{ta_1\cdot\ldots\cdot a_n}{\tau}$, say
 $a_1x_1\ge \frac{ta_1\cdot\ldots\cdot a_n}{\tau}$, which implies that
$$x_1\ge \frac{ta_2\ldots a_{\tau-1}}{\tau}\ge \frac{2^{\tau-2}\min\{a_1,\ldots, a_{\tau-1}\}}{\tau}\ge \min\{a_1,\ldots, a_{\tau-1}\}\,.$$

For each  $i\in [2,\tau-1]$, we set $x_i=y_ia_1+r_i$ with $r_i\in[0, a_1-1]$. Then
\[
a_2r_2+\ldots+a_{\tau-1}r_{\tau-1}\le (a_1-1)(a_2+\ldots +a_{\tau-1})\le a_1\cdot\ldots\cdot a_{\tau-1}=a_1\cdot\ldots\cdot a_n\,.
\]
Therefore $$a_1(x_1+y_2a_2+\ldots +y_{\tau-1}a_{\tau-1})+x_{\tau}+\ldots +x_n\ge ta_1\cdot\ldots\cdot a_n-a_1\cdot\ldots\cdot a_n\ge a_1\cdot\ldots\cdot a_n\,.$$

Suppose that  $a_1(x_1+y_2a_2+\ldots +y_{\tau-1}a_{\tau-1})<a_1\cdot\ldots\cdot a_n$. Then we can choose $x_1'=x_1$, $x_i'=a_1y_i$ for all $i\in [2, \tau-1]$, and  $x_i'\in [0, x_i]$ for all $i\in [\tau,n]$ with
$$\sum_{i=\tau}^nx_i'=a_1\cdot\ldots\cdot a_n-a_1(x_1+y_2a_2+\ldots +y_{\tau-1}a_{\tau-1})\in [1, x_{\tau}+\ldots +x_n]\,.$$
It follows that $a_1x_1'+\ldots a_nx_n'=a_1\cdot\ldots\cdot a_n$.

Suppose that  $a_1(x_1+y_2a_2+\ldots +y_{\tau-1}a_{\tau-1})\ge a_1\cdot\ldots\cdot a_n$.
If $y_2a_2+\ldots +y_{\tau-1}a_{\tau-1}\le a_2\cdot\ldots\cdot a_n$, then there exists $x_1'\in [0,x_1]$ such that $x_1'+y_2a_2+\ldots +y_{\tau-1}a_{\tau-1}=a_2\cdot\ldots\cdot a_n$ and hence $a_1x_1'+a_2y_2a_1+\ldots +a_{\tau-1}y_{\tau-1}a_1=a_1\cdot\ldots\cdot a_n$. If $y_2a_2+\ldots+ y_{\tau-1}a_{\tau-1}>a_2\cdot\ldots\cdot a_n$, then we can choose $y_i'\in [0,y_i]$ for each $i\in [2,\tau-1]$ such that $$0\le a_2\cdot\ldots\cdot a_n-(y_2'a_2+\ldots +y_{\tau-1}'a_{\tau-1})\le \min\{a_2,\ldots, a_{\tau-1}\}\le x_1\,.$$
Thus there exists $x_1'\in [0,x_1]$ such that $x_1'+y_2'a_2+\ldots+ y_{\tau-1}'a_{\tau-1}=a_2\cdot\ldots\cdot a_n$ and hence $a_1x_1'+a_2y_2'a_1+\ldots+ a_{\tau-1}y_{\tau-1}'a_1=a_1\cdot\ldots\cdot a_n$.

\bigskip
The in particular statement follows by induction on $t$.
\end{proof}

\medskip
\section{locally finitely generated monoids} \label{3}
\medskip

In this section we prove Theorems \ref{theorem1} and \ref{theorem2} formulated in the Introduction. We start with a series of lemmas.

\smallskip
\begin{lemma}\label{3.6}
Let $H$ be a monoid such that  $H_{\red}$ is finitely generated.
\begin{enumerate}
\item For every $a\in H\setminus H^{\times}$, there exists $N=N(a)\in \N$ such that
$\rho(a^{tN})=\rho(a^{N})$ for every $t\in \N$ and $\overline{\rho}(a)=\rho(a^N)$.

\item $\overline{R}(H)=\{\rho(a)\mid a\in H\setminus H^{\times} \text{ with the property that }  \rho(a^{t})=\rho(a) \text{ for all } t\in \N \}$.

\item There exists $a\in H$ such that $\rho(a)=\rho(H)=\sup \overline{R}(H)$.

\item There exists an atom $b\in \mathcal A(H)$ such that  $\lim_{n\rightarrow\infty}\rho(b^n)=\inf\overline{R}(H)$.
\end{enumerate}
\end{lemma}

\begin{proof}
1. This follows immediately by \cite[Theorem 3.8.1]{Ge-HK06a}.

2. This follows from the definition and from 1.

3. It follows from \cite[Theorem 3.1.4]{Ge-HK06a} that $H$ has accepted elasticity. Let $a\in H\setminus H^{\times}$ with $\rho(a)=\rho(H)$. Then $\lim_{n\rightarrow\infty}\rho(a^n)=\rho(H)=\sup \overline{R}(H)$.

4. If $a\in H\setminus H^{\times}$ is not an atom, then $a=u_1\cdot\ldots\cdot u_{\ell}$ where $\ell\ge 2$ and $u_i\in \mathcal A(H)$. Let $N\in \N$  such that $\lim_{n\rightarrow \infty}\rho(a^n)=\rho(a^N)$ and $\lim_{n\rightarrow \infty}\rho(u_i^n)=\rho(u_i^N)$  for every $i\in [1,\ell]$. Then $\rho(a^N)\ge \min \{\rho(u_i^N)]\mid i\in [1,\ell]\}$. It follows by $\mathcal A(H)$ is finite that
 $$\inf \overline{R}(H)=\inf \left\{\lim_{n\rightarrow\infty}\rho(u^n)\mid u\in \mathcal A(H)\right\}=\min \left\{\lim_{n\rightarrow\infty}\rho(u^n)\mid u\in \mathcal A(H)\right\}\,. \qedhere
 $$
\end{proof}

\begin{definition}
Let $H$ be an atomic monoid and let $a,b\in H\setminus H^{\times}$. We say the pair $(k,\ell) \in \N_0^2\setminus\{(0,0)\}$ is nice with respect to $(a,b)$ if  for every $t\in \N$
\[
\begin{aligned}
\max\mathsf L((a^{k}b^{\ell})^t)&=t\max\mathsf L(a^{k}b^{\ell}) \quad  \text{ and }  \quad
\min\mathsf L((a^{k}b^{\ell})^t)&=t\min\mathsf L(a^{k}b^{\ell})\,.
\end{aligned}\]
It is easy to check that $(k, \ell)$ is nice if and only if  $\rho((a^kb^{\ell})^t)=\rho(a^kb^{\ell})$ for all $t\in \N$.
\end{definition}

\begin{lemma}
Let $H$ be an atomic monoid, $a,b\in H\setminus H^{\times}$, and let $(k,\ell) \in \N_0^2\setminus\{(0,0)\}$ be a nice pair with respect to $(a,b)$.
\begin{enumerate}
\item $(tk, t\ell)$ is a nice pair with respect to $(a,b)$ for each $t\in \N$.

\item $\rho(a^kb^\ell)\in \overline{R}(H)$.

\item For every $x\in \Q$ with $x\ge 0$, there exists a nice pair $(k', \ell')$ with respect to $(a,b)$ such that $\ell'/k'=x$, where $k'\in \N$ and $\ell'\in \N_0$.
\end{enumerate}

\end{lemma}
\begin{proof}
1. It is obvious by definition.

2. It follows immediately by Lemma \ref{3.6}.2.

3. Let $x\in \Q$ with $x\ge 0$ and  let $m,n\in \N_0$ such that $x=m/n$. Then $a^mb^n\in H\setminus H^{\times}$. It follows by Lemma \ref{3.6}.1 that there exists $N\in \N$ such that $\rho((a^{Nm}b^{Nn})^t)=\rho(a^{Nm}b^{Nn})$ for all $t\in \N$.  Therefore $(Nm,Nn)$ is a nice pair with respect to $(a,b)$ such that $Nm/Nn=x$.
\end{proof}

\begin{proposition}\label{3.8}
Let $H$ be a monoid such that  $H_{\red}$ is finitely generated and $\min\overline{R}(H)<\max \overline{R}(H)$. Let $c\in H$ with $\lim_{n\rightarrow \infty}\rho(c^n)=\max \overline{R}(H) $ and let $b\in \mathcal A(H)$ with $\lim_{n\rightarrow \infty}\rho(b^n)=\min\overline{R}(H)$.

\begin{enumerate}

\item There exists $M\in \N$ satisfying that  for every $k\in \N$ and every $\ell\in \N_0$, there exist $\ell_1,\ldots, \ell_k\in \Z$ with $\ell_1+\ldots+\ell_k=\ell$  and all $ c^Mb^{\ell_i}\in H$ such that
\begin{align*}
\max\mathsf L(c^{kM}b^{\ell})=\sum_{i=1}^{k}\max\mathsf L(c^Mb^{\ell_i})\ \  \text{ and }\ \
\min\mathsf L(c^{kM}b^{\ell})=\sum_{i=1}^{k}\min\mathsf L(c^Mb^{\ell_i})\,.
\end{align*}


\item  Set $a=c^M$.
Let $(k,\ell)\in \N_0^2\setminus\{(0,0)\}$ be a nice pair with respect to $(a,b)$ and
 let $\ell_1,\ldots, \ell_{k}\in \Z$ with $\ell_1+\ldots+\ell_{k}=\ell$ and all $ab^{\ell_i}\in H$ such that
\begin{align*}
\max\mathsf L(a^{k}b^{\ell})=\sum_{i=1}^{k}\max\mathsf L(ab^{\ell_i}) \text{ and }
\min\mathsf L(a^{k}b^{\ell})=\sum_{i=1}^{k}\min\mathsf L(ab^{\ell_i})\,.
\end{align*}

\begin{enumerate}

\item For distinct $i_1,i_2\in [1, k]$ and any $t_1,t_2\in \N_0$, we have \[\begin{aligned}
\max\mathsf L((ab^{\ell_{i_1}})^{t_1}(ab^{\ell_{i_2}})^{t_2})&=t_1\max\mathsf L(ab^{\ell_{i_1}})+t_2\max\mathsf L(ab^{\ell_{i_2}})\,,\\
\text{and }\quad \min\mathsf L((ab^{\ell_{i_1}})^{t_1}(ab^{\ell_{i_2}})^{t_2})&=t_1\min\mathsf L(ab^{\ell_{i_1}})+t_2\min\mathsf L(ab^{\ell_{i_2}})\,.
\end{aligned}\]
In particular, $(1, \ell_i)$ is a nice pair with respect to $(a, b)$ for each $i\in [1,k]$.

\item Let $\tau$  be the maximal non-negative integer $\ell$ such that $\rho(\mathsf L(ab^{\ell}))=\rho(\mathsf L(a))$. If $\ell/k>\tau$, then $\ell_j\ge \tau$ for every $j\in [1,k]$.
\end{enumerate}

\item Let $I$ be the set of all $t\in \N_{\ge \tau}$ such that $(1, t)$ is a nice pair respect to $(a,b)$.
      \begin{enumerate}
      \item $I$ is infinite.

      \item Suppose $I=\{t_1,t_2,\ldots\}\subset \N_{\ge \tau}$ with $\tau=t_1<t_2<\ldots$. Then $\rho(ab^{t_j})\ge \rho(ab^{t_{j+1}})$ for all $j\in \N$ and $\lim_{j\rightarrow \infty}\rho(ab^{t_j})=\inf \overline{R}(H)$.
      \end{enumerate}

\item Let $i\in \N$, let $x\in \Q$ with $t_i<x<t_{i+1}$, and let  $k, \ell\in\N$ with $\ell/k=x$  be such that $(k, \ell)$ is  a nice pair  with respect to $(a,b)$. Then  $$\rho(a^kb^{\ell})=\frac{(t_{i+1}-x)\max\mathsf L(ab^{t_{i+1}})+(x-t_i)\max\mathsf L(ab^{t_i})}{(t_{i+1}-x)\min\mathsf L(ab^{t_{i+1}})+(x-t_i)\min\mathsf L(ab^{t_i})}\,.$$

\item $\overline{R}(H)=\{q\in\Q\mid \inf \overline{R}(H)\le q\le \sup \overline{R}(H)\}$.
\end{enumerate}
\end{proposition}

\begin{proof}
We may assume that $H$ is reduced. We denote by $\mathsf Z(H):=\mathcal F(\mathcal A(H))$ the factorization monoid of $H$ and by $\pi:\mathsf Z(H)\rightarrow H$ the factorization homomorphism. If $z \in \mathsf Z (H)$, then $|z| = |z|_{\mathcal F ( \mathcal A (H))}$ denotes the length of $z$.

1. Assume to the contrary that there exists $n\in \N$ such that $c\t b^n$. Then $b^n=cd$ for some $d\in H\setminus H^{\times}$. Let $N\in \N$ such that $\lim_{m\rightarrow\infty}\rho(c^m)=\rho(c^{N})$, $\lim_{m\rightarrow\infty}\rho(b^m)=\rho(b^{nN})$, and $\lim_{m\rightarrow\infty}\rho(d^m)=\rho(d^{N})$. Thus $\rho(b^{nN})\ge \min \{\rho(d^N), \rho(c^N)\}\ge \rho(d^N)\ge \rho(b^{nN})$ which implies that $\rho(b^{nN})=\rho(d^N)=\rho(c^N)$, a contradiction.
Thus $c\not|\ b^n$ for every $n\in \N$.

Let
$$H_0=\{(x,y)\in \mathsf Z(H)\times \mathsf Z(H)\mid \pi(x)=\pi(y)=c^kb^{\ell}\in H \text{ for some $k\in \N_0$ and some $\ell\in \Z$}\}\,.$$
Then $H_0$ is a saturated submonoid of $\mathsf Z(H)\times \mathsf Z(H)$.
Since $H$ is finitely generated, it follows by \cite[Proposition 2.7.5]{Ge-HK06a} that  $H_0$ is finitely generated.

Suppose  $\mathcal A(H_0)=\{(x_1, y_1), \ldots, (x_t, y_t)\}$ with $\pi(x_i)=\pi(y_i)=c^{k_i}b^{\ell_i}\in H$ for each $i\in [1,t]$,
 where $t\in \N, k_1,\ldots, k_t\in \N_0$ and $\ell_1,\ldots, \ell_t\in \Z$.
Note that $(b,b)\in H_0$ is an atom of $H_0$. We obtain that $\min\{k_1,\ldots,k_t\}=0$. After renumbering if necessary, we assume that there exists $t_0\in[1, t]$ such that $k_i\ge 1$ for all $i\in [1,t_0]$ and $k_i=0$ for all $i\in [t_0, t]$.

Let $M=\prod_{i=1}^{t_0}k_i$ and let $k\in \N$, $\ell\in \N_0$.  We choose  $(x,y)\in \mathsf Z(H)\times \mathsf Z(H)$ with $\pi(x)=\pi(y)=c^{kM}b^{\ell}$ such that $|x|=\min \mathsf L(c^{kM}b^{\ell})$ and $|y|=\max \mathsf L(c^{kM}b^{\ell})$.

Suppose $(x,y)=\prod_{i=1}^t(x_i,y_i)^{v_i}$, where $v_i\in \N_0$. If $v_i\neq 0$, then $|x_i|=\min \mathsf L(c^{k_i}b^{\ell_i})$ and $|y_i|=\max \mathsf L(c^{k_i}b^{\ell_i})$.
Therefore
\begin{align*}
\max\mathsf L(c^{kM}b^{\ell})&=\sum_{i=1}^tv_i|y_i|=\sum_{i=1}^tv_i\max\mathsf L(c^{k_i}b^{\ell_i})\,,\\
\min\mathsf L(c^{kM}b^{\ell})&=\sum_{i=1}^tv_i|x_i|=\sum_{i=1}^tv_i\min\mathsf L(c^{k_i}b^{\ell_i})\,,\\
\sum_{i=1}^t k_iv_i&=kM \quad \text{ and } \quad \sum_{i=1}^{t} \ell_iv_i=\ell\,.
 \end{align*}

Since $$kM=k\prod_{i=1}^{t_0}k_i=\sum_{i=1}^tv_ik_i=\sum_{i=1}^{t_0}v_ik_i\,,$$
 it follows by Lemma \ref{3.2} that there exist $x_i^{(j)}\in [0, v_i], i\in [1,t_0], j\in [1,k]$ such that $\sum_{j=1}^kx_i^{(j)}=v_i$ for every $i\in [1,t_0]$ and $\sum_{i=1}^{t_0}x_i^{(j)}k_i=M$ for every $j\in [1,k]$. Let
\begin{align*}
\ell_1'&=\sum_{i=1}^{t_0}x_i^{(1)}\ell_i+\sum_{i=t_0+1}^tv_i\ell_i\\
 \text{and }\quad
 \ell_j'&=\sum_{i=1}^{t_0}x_i^{(j)}\ell_i\quad \text{ for every }j\in [2,k]\,.
\end{align*}

  Then  $\ell_1'+\ldots+\ell_k'=\sum_{i=1}^{t}v_i\ell_i=\ell$ and
\begin{align*}
\max\mathsf L(c^{kM}b^{\ell})\ge \sum_{j=1}^k\max\mathsf L(c^{M}b^{\ell_j'}) &\ge \sum_{i=1}^tv_i\max\mathsf L(c^{k_i}b^{\ell_i})=\max\mathsf L(c^{kM}b^{\ell})\,,\\
\min\mathsf L(c^{kM}b^{\ell})\le \sum_{j=1}^k\min\mathsf L(c^{M}b^{\ell_j'}) &\le \sum_{i=1}^tv_i\max\mathsf L(c^{k_i}b^{\ell_i})=\min\mathsf L(c^{kM}b^{\ell})\,.
\end{align*}

2. (a)  Without loss of generality, we assume that $i_1=1$ and $i_2=2$. Let $t=\max\{t_1,t_2\}$. Then
$$\begin{aligned}
&\quad t\sum_{j=1}^{k}\max\mathsf L(ab^{\ell_i})=\max\mathsf L((a^{k}b^{\ell})^t)\\
&\ge\max \mathsf L\big((ab^{\ell_1})^{t_1}(ab^{\ell_2})^{t_2}\big)+
(t-t_1)\max\mathsf L(ab^{\ell_1})+(t-t_2)\max\mathsf L(ab^{\ell_2})+t\sum_{j=3}^{k}\max\mathsf L(ab^{\ell_i})\\
&\ge t\sum_{j=1}^{k}\max\mathsf L(ab^{\ell_i})\,.
\end{aligned} $$
It follows that $$\max \mathsf L\big((ab^{\ell_1})^{t_1}(ab^{\ell_2})^{t_2}\big)=t_1\max\mathsf L(ab^{\ell_1})+t_2\max\mathsf L(ab^{\ell_2})\,.$$
By the similar argument, we can obtain $$\min \mathsf L\big((ab^{\ell_1})^{t_1}(ab^{\ell_2})^{t_2}\big)=t_1\min\mathsf L(ab^{\ell_1})+t_2\min\mathsf L(ab^{\ell_2})\,.$$

\medskip
(b) Suppose $\ell/k>\tau$.
 Since $\sum_{i=1}^{k}\ell_i=\ell>\tau$,
there exists some $i\in [1,k]$, say $i=1$, such that $\ell_1>\tau$. Thus $\rho(\mathsf L(ab^{\ell_1}))<\rho(\mathsf L(a))$ by the definition of $\tau$.

Assume to the contrary that there exists some $i\in [1,k]$, say $i=2$, such that  $\ell_2<\tau$.
Then $$\begin{aligned}
\rho(\mathsf L(a))&=\rho(\mathsf L(ab^{\tau}))=\rho(\mathsf L((ab^{\tau})^{\ell_1-\ell_2}))\\
&=\frac{\max\mathsf L(a^{(\ell_1-\ell_2)}b^{\tau(\ell_1-\ell_2)})}{\min\mathsf L(a^{(\ell_1-\ell_2)}b^{\tau(\ell_1-\ell_2)})}\\
&=\frac{\max\mathsf L\big((ab^{\ell_1})^{\tau-\ell_2}(ab^{\ell_2})^{\ell_1-\tau}\big)}{\min\mathsf L\big((ab^{\ell_1})^{\tau-\ell_2}(ab^{\ell_2})^{\ell_1-\tau}\big)}\\
&=\frac{(\ell_1-\tau)\max\mathsf L(ab^{\ell_2})+(\tau-\ell_2)\max\mathsf L(ab^{\ell_1})}{(\ell_1-\tau)\min\mathsf L(ab^{\ell_2})+(\tau-\ell_2)\min\mathsf L(ab^{\ell_1})}\\
&<\rho(\mathsf L(a))\,,
\end{aligned}$$
a contradiction.

3. (a)
Assume to the contrary that $I$ is finite. Let $r=\max I$ and
let $k,\ell\in \N$ with  $\ell/k>r\ge \tau$  such that $(k,\ell)$ is a nice pair with respect to $(a,b)$.
Then  there exist $\ell_1, \ell_2,\ldots, \ell_{k}\in \Z$ with $\ell_1+\ldots+\ell_{k}=\ell$ and
\begin{align*}
\max\mathsf L(a^{k}b^{\ell})=\sum_{i=1}^{k}\max\mathsf L(ab^{\ell_i})\ \  \text{ and }\ \
\min\mathsf L(a^{k}b^{\ell})=\sum_{i=1}^{k}\min\mathsf L(ab^{\ell_i})\,.
\end{align*}
It follows by Lemma \ref{3.8}.2 that $\ell_i\ge \tau$ and $\ell_i\in I$ for all $i\in [1,k]$.
 Since $\sum_{i=1}^{k}\ell_i=\ell>kr$,
there exists some $i\in [1,k]$, say $i=1$, such that $\ell_1>r$. We obtain a contradiction to $\ell_1\in I$ and $\max I=r$.

(b)
Let $I=\{t_1, t_2, \ldots\}$ with $\tau=t_0<t_1<\ldots$.
For every $i\in \N$, we have that
\begin{align*}
\rho(\mathsf L(ab^{t_i}))&=\rho(\mathsf L((ab^{t_i})^{t_{i+1}-t_1}))=\rho(\mathsf L((ab^{t_1})^{t_{i+1}-t_i}(ab^{t_{i+1}})^{t_i-t_1}))\\
&\ge\frac{(t_{i+1}-t_i)\max\mathsf L(ab^{t_1})+(t_i-t_1)\max\mathsf L(ab^{t_{i+1}})}{(t_{i+1}-t_i)\min\mathsf L(ab^{t_1})+(t_i-t_1)\min\mathsf L(ab^{t_{i+1}})}\\
&\ge \min \{\rho(\mathsf L(ab^{t_1})), \rho(\mathsf L(ab^{t_{i+1}}))\} \\
&=\rho(\mathsf L(ab^{t_{i+1}}))\,.
\end{align*}

In order to show $\lim_{j\rightarrow\infty}\rho(ab^{t_j})=\inf\overline{R}(H)$, we prove that $\lim_{\ell\rightarrow\infty}\rho(ab^{\ell})=\inf\overline{R}(H)$.

Note that $a\not| b^n$ for any $n\in\N$.
Let
$$H_1=\{(x,y)\in \mathsf Z(H)\times \mathsf Z(H)\mid \pi(x)=\pi(y)=a^kb^{\ell}\in H \text{ for some $k\in \N_0$ and some $\ell\in \Z$}\}\,.$$
Then $H_1$ is a saturated submonoid of $\mathsf Z(H)\times \mathsf Z(H)$.
Since $H$ is finitely generated, we obtain $H_1$ is finitely generated.
Suppose  $\mathcal A(H_1)=\{(x_1, y_1), \ldots, (x_t, y_t)\}$ with $\pi(x_i)=\pi(y_i)=a^{k_i}b^{\ell_i}\in H$ for each $i\in [1,t]$,
 where $t\in \N, k_1,\ldots, k_t\in \N_0$ and $\ell_1,\ldots, \ell_t\in \Z$.
Since $(z_1,z_2)\in \mathcal A(H_1)$ for any $z_1,z_2\in \mathsf z(a)$, we have that there exists $i\in [1,t]$ such that $k_i=1$.
Let $\ell_{\max}=\max\{\ell_i\mid k_i=1\}$ and  $\ell_{\min}=\min\{\ell_i\mid k_i=1\}$.

For every $\ell\ge \ell_{\max}$, we let $(x,y)\in H_1$ with $\pi(x)=ab^{\ell}$ such that $|x|=\min \mathsf L(ab^{\ell})$ and $|y|=\max \mathsf L(ab^{\ell})$.
Suppose $(x,y)=\prod_{i=1}^{t}(x_i,y_i)^{v_i}$, where all $v_i\in \N_0$. Therefore $1=\sum_{i=1}^t k_iv_i$ and $\ell=\ell_iv_i $. It follows  that there exists $\ell'\in [\ell_{\min}, \ell_{\max}]$ such that
\begin{align*}
\max\mathsf L(ab^{\ell})&=\max\mathsf L(ab^{\ell'})+\max\mathsf L(b^{\ell-\ell'})\le \max\mathsf L(ab^{\ell_{\max}})+\max\mathsf L(b^{\ell-\ell_{\min}})\\
\min\mathsf L(ab^{\ell})&=\min\mathsf L(ab^{\ell'})+\min\mathsf L(b^{\ell-\ell'})\ge \min\mathsf L(ab^{\ell_{\min}})+\min\mathsf L(b^{\ell-\ell_{\max}})\\
&\ge \min\mathsf L(ab^{\ell_{\min}})+\min\mathsf L(b^{\ell-\ell_{\min}})-\min\mathsf L(b^{\ell_{\max}-\ell_{\min}})
\end{align*}

If follows that $$\rho(b^{\ell})\le \rho(ab^{\ell})\le \frac{\max\mathsf L(ab^{\ell_{max}})+\max\mathsf L(b^{\ell-\ell_{min}})}{\min\mathsf L(ab^{\ell_{min}})-\min\mathsf L(b^{\ell_{max}-\ell_{min}})+\min\mathsf L(b^{\ell-\ell_{min}})}\,. $$
Since $$\lim_{\ell\rightarrow\infty}\frac{\max\mathsf L(ab^{\ell_{max}})+\max\mathsf L(b^{\ell-\ell_{min}})}{\min\mathsf L(ab^{\ell_{min}})-\min\mathsf L(b^{\ell_{max}-\ell_{min}})+\min\mathsf L(b^{\ell-\ell_{min}})}=\lim_{\ell\rightarrow\infty}\frac{\max\mathsf L(b^{\ell-\ell_{min}})}{\min\mathsf L(b^{\ell-\ell_{min}})}=\inf \overline{R}(H)\,,$$
we obtain that $\lim_{\ell\rightarrow\infty}\rho(ab^{\ell})=\inf \overline{R}(H)$.

4. Let $j\in \N$ and $x\in \Q$ with $t_j<x<t_{j+1}$. There exist $k,\ell\in \N$ with $\ell/k=x$ such that $(k,\ell)$ is a nice pair with respect to $(a,b)$.
We fix such a pair. Then for any $t\in \N$, $(tk, t\ell)$ is a nice pair with respect to $(a,b)$.
By 1. and 2.b., we choose $\ell_1,\ldots, \ell_{tk}\in \N_{\ge \tau}$ with $\sum_{i=1}^{tk} \ell_i=t\ell$  and
\begin{align*}
\max\mathsf L(a^{tk}b^{t\ell})=\sum_{i=1}^{tk}\max\mathsf L(ab^{\ell_i}) \text{ and }
\min\mathsf L(a^{tk}b^{t\ell})=\sum_{i=1}^{tk}\min\mathsf L(ab^{\ell_i})\,,
\end{align*}
such that $$C_t=\big|I\cap [\min\{\ell_i\mid i\in[1,tk]\}, \max\{\ell_i\mid i\in[1,tk]\}]\big|$$ is minimal. Let $N\in \N$ such that $C_{N}=\min \{C_t\mid t\in \N\}$ and $\ell_1,\ldots, \ell_{Nk}$ are the corresponding non-negative integers.
It follows by Lemma \ref{3.8}.3.a that  $\ell_i\in I$ for all $i\in [1,Nk]$.

Without loss of generality, we let $\ell_1=\max\{\ell_i\mid i\in [1,Nk]\}$ and $\ell_2=\min\{\ell_i\mid i\in [1,Nk]\}$.
Assume to the contrary that $C_{N}\ge 3$. Then there is a $y\in I$ such that $\ell_2<y<\ell_1$.
Let $S=\ell_1\cdot\ldots\cdot \ell_{Nk}$ be the sequence over $\N_{\ge \tau_M}$. If $\mathsf v_{\ell_1}(S)(\ell_1-y)\le \mathsf v_{\ell_2}(S)(y-\ell_2)$, then we choose $\ell_1'\cdot\ldots\cdot \ell_{Nk(y-\ell_2)}'=\ell_2^{(y-\ell_2)\mathsf v_{\ell_2}(S)-\mathsf v_{\ell_1}(S)(\ell_1-y)}y^{(\ell_1-\ell_2)\mathsf v_{\ell_1}(S)}\prod_{i=3}^{Nk}\ell_i^{y-\ell_2}$ and hence $\sum_{i=1}^{Nk(y-\ell_2)}\ell_i'=N\ell(y-\ell_2)$.
Since
\begin{align*}
&(\ell_1-\ell_2)\max\mathsf L(ab^y)=
\max\mathsf L((ab^y)^{\ell_1-\ell_2})=\\
&\qquad \max\mathsf L((ab^{\ell_1})^{y-\ell_2}(ab^{\ell_2})^{\ell_1-y})=
(y-\ell_2)\max\mathsf L(ab^{\ell_1})+(\ell_1-y)\max\mathsf L(ab^{\ell_2})\\
\text{and}&\\
&(\ell_1-\ell_2)\min\mathsf L(ab^y)=
\min\mathsf L((ab^y)^{\ell_1-\ell_2})=\\
&\qquad \min\mathsf L((ab^{\ell_1})^{y-\ell_2}(ab^{\ell_2})^{\ell_1-y})=
(y-\ell_2)\min\mathsf L(ab^{\ell_1})+(\ell_1-y)\min\mathsf L(ab^{\ell_2})\,,
\end{align*}
we have
\begin{align*}
&\max\mathsf L(a^{Nk(y-\ell_2)}b^{N\ell (y-\ell_2)})=(y-\ell_2)\sum_{i=1}^{Nk}\max\mathsf L(ab^{\ell_i})=\sum_{i=1}^{Nk(y-\ell_2)}\max\mathsf L(ab^{\ell_i'})\\
&\min\mathsf L(a^{Nk(y-\ell_2)}b^{N\ell (y-\ell_2)})=(y-\ell_2)\sum_{i=1}^{Nk}\min\mathsf L(ab^{\ell_i})=\sum_{i=1}^{Nk(y-\ell_2)}\min\mathsf L(ab^{\ell_i'})\,.
\end{align*}
But $C_{N(y-\ell_2)}<C_{N}$, a contradiction to the minimality of $C_{N}$.
If $\mathsf v_{\ell_1}(S)(\ell_1-y)\ge \mathsf v_{\ell_2}(S)(y-\ell_2)$, then we can get a contradiction similarly.

Therefore  $C_{N}\le 2$.
  Note that
\[
\ell_2 = \min \{\ell_i \mid i \in [1, Nk]\} \le \frac{\sum_{i=1}^{Nk}\ell_i}{Nk} = \frac{\ell}{k} \le \max \{\ell_i \mid i \in [1, Nk]\} = \ell_1 \,.
\]
If $\ell_1=\ell_2$, then $\ell/k = \ell_1$, a contradiction to $t_j<\ell/k<t_{j+1}$. Thus we get that $\ell_2 < \ell/k < \ell_1$. Since
 $t_j$ is the maximal element of $I$ that is smaller than $\ell/k$ and $t_{j+1}$ is the minimal element of $I$ that is larger than  $\ell/k$, we obtain that $\ell_2 \le t_j < t_{j+1} \le \ell_1$ whence $\{t_j, t_{j+1} \} \subset I_M \cap [\ell_2, \ell_1]$. Since $C_{N} \le 2$ and $\{\ell_i \mid i \in [1, Nk] \} \subset I \cap [\ell_2, \ell_1]$, it follows that
\begin{equation} \label{crucial2}
\{\ell_i\mid i\in [1,Nk]\}= I \cap [\ell_2, \ell_1] = \{t_j, t_{j+1}\} \,.
\end{equation}
Then
\[
\begin{aligned}
\rho(a^{Nk}b^{N\ell}) & = \frac{\max \mathsf L ( a^{Nk}b^{N\ell} )}{\min \mathsf L ( a^{Nk}b^{N\ell} )} \\
& = \frac{\sum_{i=1}^{Nk} \max \mathsf L ( ab^{\ell_i})}{\sum_{i=1}^{Nk} \min \mathsf L (a b^{\ell_i})}  \\
& =\frac{x_1\max\mathsf L(ab^{t_{j}})+x_2\max\mathsf L(ab^{t_{j+1}})}{x_1\min\mathsf L(ab^{t_{j}})+x_2\min\mathsf L(ab^{t_{j+1}})} \,,
\end{aligned}
\]
where by \eqref{crucial2}
\[
x_1 = |\{i \in [1, Nk] \mid \ell_i = t_j\}| \quad \text{and} \quad x_2 = |\{i \in [1, Nk] \mid \ell_i = t_{j+1}\}| \,.
\]
Comparing exponents of $a$ and $b$ we obtain the equations
\[
x_1t_j + x_2t_{j+1} = N\ell  \quad \text{and} \quad x_1 + x_2 = Nk
\]
whence
\[
(x_1,x_2)= \left(\frac{(kt_{j+1}-\ell)N}{t_{j+1}-t_j}, \frac{(\ell-kt_j)N}{t_{j+1}-t_j} \right) \,.
\]
Plugging in this expression for $(x_1, x_2)$ we obtain that
\[
\begin{aligned}
\rho(a^kb^{\ell})=\rho(a^{Nk}A_0^{N\ell}) & =
 \frac{ (k t_{j+1} - \ell) \max\mathsf L(ab^{t_{j}}) + (\ell - kt_j)\max\mathsf L(ab^{t_{j+1}})}{(k t_{j+1} - \ell) \min \mathsf L(ab^{t_{j}}) + (\ell - kt_j)\min \mathsf L(ab^{t_{j+1}})} \\
& =  \frac{ ( t_{j+1} - x) \max\mathsf L(ab^{t_{j}}) + (x - t_j)\max\mathsf L(ab^{t_{j+1}})}{(t_{j+1} - x) \min \mathsf L(ab^{t_{j}}) + (x - t_j)\min \mathsf L(ab^{t_{j+1}})} \,.
\end{aligned}
\]

5. It follows by Lemma \ref{3.6}.2 that
$\overline{R}(H)\subset \{q\in \Q\mid \inf \overline{R}(H)\le q\le \sup\overline{R}(H)\}$.

 By 3., we know that $\rho(ab^{t_j})\in \overline{R}(H)$ for each $j\in \N$ and $\rho(H)\in \overline{R}(H)$. If suffices to prove that $\{q\in \Q\mid \rho(ab^{t_{j+1}})<q<\rho(ab^{t_{j}}))\}\subset \overline{R}(H)$.

By 4., we obtain that for each $j\in \N$,
\begin{align*}
&\{q\in \Q\mid \rho(ab^{t_{j+1}})<q<\rho(ab^{t_{j}}))\}\\
=&\left\{\frac{ ( t_{j+1} - x) \max\mathsf L(ab^{t_{j}}) + (x - t_j)\max\mathsf L(ab^{t_{j+1}})}{(t_{j+1} - x) \min \mathsf L(ab^{t_{j}}) + (x - t_j)\min \mathsf L(ab^{t_{j+1}})}\mid x\in \Q \text{ and } t_j<x<t_{j+1}\right\}\\
\subset& \{\rho(a^kb^{\ell})\mid (k,\ell) \text{  is a nice pair with respcet to }(a,b) \text{ such that }t_j<\ell/k<t_{j+1}\}\\
\subset &\overline{R}(H)\,. \qedhere
\end{align*}
\end{proof}

\medskip
\noindent{\bf Proof of Theorem \ref{theorem1}. }
Suppose that $H$ is a locally finitely generated monoid. For every $q\in \Q$ with $\sup\overline{R}(H)>q>\inf\overline{R}(H)$, there exist $a,b\in H\setminus H^{\times}$ such that $\lim_{n\rightarrow \infty}\rho(a^n)>q>\lim_{n\rightarrow \infty}\rho(b^n)$.

Let $S=\LK ab\RK$. Then $S_{\red}$ is finitely generated. It follows  by  Proposition \ref{3.8}.5 that
$$q\in \left\{p\in \Q\mid \lim_{n\rightarrow \infty}\rho(a^n)>p>\lim_{n\rightarrow \infty}\rho(b^n)\right\}\subset \overline{R}(S)\subset \overline{R}(H)\,.$$
Therefore $H$ is asymptotic fully elastic.

If $\min \overline{R}(H)=1$, then it follows by  Proposition \ref{3.8}.5 and Lemma \ref{3.6}.2 that
$$\overline{R}(H)=\{q\in \Q\mid 1\le q\le \rho(H)\}\subset \{\rho(L )\mid  L \in \mathcal L (H) \}\,.$$
Therefore $H$ is fully elastic.
\qed

\begin{proof}[\bf Proof of Theorem \ref{theorem2}]
Without restriction we may suppose that  $H$ is reduced and we set  $\mathcal A(H)=\{u_1, \ldots, u_m\}$.
Proposition \ref{3.8}.5 and Lemma \ref{3.6}.2 imply that
\[
\overline{R}(H)=\{q\in \Q\mid 1\le q\le \rho(H)\}\subset \{\rho(L) \mid L \in \mathcal L (H) \}\,.
\]
Thus it remains to prove that, for a given $\epsilon>0$,  the set
\[
\{\rho(L)\mid L\in \mathcal L(H)\text{ and  }1\le \rho(L)\le r-\epsilon\}
\]
is finite. Assume to the contrary that this set is infinite. Then there is a sequence $(a_k)_{k=1}^{\infty}$ with elements from $H$ such that  $1\le\rho(a_k)\le r-\epsilon$ and $(\rho(a_k))_{k=1}^{\infty}$ is strictly increasing or decreasing. For each $k$, we fixed a factorization $z_k=u_1^{t_k^{(1)}}\cdot\ldots\cdot u_m^{t_k^{(m)}}$ of $a_k$ with $|z_k|=\min \mathsf L(a_k)$. Let $I\subset[1,m]$ be the set of all $i\in [1,m]$ such that $\{t_k^{(i)}\mid k\in \N\}$ is finite and let $M\in \N$ such that $t_k^{(i)}\le M$ for all $k\in \N$ and $i\in I$. For each $j\in [1,m]\setminus I$, let $\rho_j=\lim_{n\rightarrow\infty}\rho(u_j^n)$ and let $N\in \N$ such that $\rho_j=\rho(u_j^N)$ for all $j\in [1,m]\setminus I$.
Then for all $k\ge N$, we have $\max\mathsf L(u_j^k)=k\rho_j$ for all $j\in [1,m]\setminus I$ by the minimality of $|z_k|$ and $\{t_k^{(j)}\}$ is infinity.

It follows that  $$\rho(a_k)\ge \frac{\sum_{j\in [1,m]\setminus I}t_k^{(j)}\rho_j}{M(m-|I|)+\sum_{j\in [1,m]\setminus I}t_k^{(j)}}$$ and hence $$\lim_{k\rightarrow\infty}\rho(a_k)\ge \lim_{k\rightarrow\infty}\frac{\sum_{j\in [1,m]\setminus I}t_k^{(j)}\rho_j}{\sum_{j\in [1,m]\setminus I}t_k^{(j)}}\ge r\,,$$
 a contradiction.
\end{proof}

\begin{proposition}\label{4.1}
If  $(H_i)_{i\in \N}$ is a family of   \BF-monoids with the same elasticity, then their coproduct $\coprod_{i=1}^{\infty}H_i$ is fully elastic.
 \end{proposition}

\begin{proof}
Let $H=\coprod_{i=1}^{\infty}H_i$ and $r=\rho(H_1)$.
It is clear that $\rho(H)\ge r$. Let $a\in H$. Then there exist a finite subset $I\subset \N$  and $b_i\in H_i$ for each $i\in I$ such that $a=\prod_{i\in I}b_i$. Therefore
\[\mathsf L(a)=\sum_{i\in I}\mathsf L(b_i)
\]
and hence $$\rho(a)=\frac{\max\mathsf L(b_i)}{\min\mathsf L(b_i)}\le \max\{\rho(b_i)\mid i\in I\}\le r\,.$$
It follows that $\rho(H)=r$.

Let $q=\frac{m}{n}\in \Q$ with $1<q<\rho(H)$, where $m,n\in \N$. Then for each $i\in \N$, there exists $a_i\in H_i$ such that $\rho(a_i)=\frac{m_i}{n_i}\ge q$, where $m_i,n_i\in \N$. There exist $d\in [0,m-n-1]$ and $J\subset \N$ with $|J|=m-n$ such that $m_j-n_j\equiv d\pmod {m-n}$. Therefore $m-n\t \sum_{j\in J}(m_j-n_j)$. Since $\frac{\sum_{j\in J}m_j}{\sum_{j\in J}n_j}\ge \frac{m}{n}$, we have  $$x=\frac{n\sum_{j\in J}m_j-m\sum_{j\in J}n_j}{m-n}=\frac{n\sum_{j\in J}(m_j-n_j)-(m-n)\sum_{j\in J}n_j}{m-n}\in \N\,.$$
It follows that $q=\frac{x+\sum_{j\in J}m_j}{x+\sum_{j\in J}n_j}$

 Let $J_1\subset \N$ with $|J_1|=x$, $J_1\cap J=\emptyset$ and let $c_i\in\mathcal A(H_i)$ for each $i\in J_1$.  Therefore $\rho(\prod_{j\in J}b_j\prod_{i\in J_1}c_i)=\frac{x+\sum_{j\in J}m_j}{x+\sum_{j\in J}n_j}=q$ and hence $H$ is fully elastic.
\end{proof}

\begin{example}\label{ex}
Let $H$ be a finitely generated monoid with $\min\overline{R}(H)=r>1$ (Note that numerical monoids, distinct from $(\N_0,+)$, have this property and they are not fully elastic). Then the coproduct $H^*=\coprod_{i=1}^{\infty}H_i$, with $H_i=H$ for all $i \in \N$,  is locally finitely generated. Since $\inf\overline{R}(H^*)=\inf \{\min \overline{R}(H_i)\mid i\in \N\}$, it follows  that $\min \overline{R}(H^*)=r$. By Proposition \ref{4.1}, we infer  that $H^*$ is fully elastic.
\end{example}

\medskip
\section{Weakly Krull monoids} \label{4}

Let $H$ be a  monoid. We denote  by $\mathfrak X (H)$ the set of minimal nonempty prime $s$-ideals of $H$,  and by $\mathfrak m = H \setminus H^{\times}$ the maximal $s$-ideal. Let $\mathcal I_v^* (H)$ denote the monoid of $v$-invertible $v$-ideals of $H$ (with $v$-multiplication). Then $\mathcal F_v (H)^{\times} = \mathsf q ( \mathcal I_v^* (H))$ is the quotient group of fractional  $v$-invertible $v$-ideals, and $\mathcal C_v (H) = \mathcal F_v (H)^{\times}/\{ xH \mid x \in \mathsf q (H)\}$ is the $v$-class group of $H$ (detailed presentations of ideal theory in commutative monoids can be found in \cite{HK98, Ge-HK06a}).

We denote by
\begin{itemize}

\item $\widehat{H}=\{a\in \mathsf q(H)\mid \text{ there exists } c\in H \text{ such that } ca^n\in H \text{ for all }n\in \N\}$ the completely integral closure of $H$,

\item $(H \DP \widehat H) = \{ x \in \mathsf q (H) \mid x \widehat H \subset H \} \subset H$  the conductor of $H$, and

\item $H'=\{a\in \mathsf q(H)\mid \text{ there exists }N\in \N \text{ such that } a^n\in H \text{ for all } n\ge N\}$ the seminormal closure of $H$.
\end{itemize}
We say $H$ is  completely integrally closed if $H=\widehat{H}$ and $H$ is seminormal if $H'=H$.

\medskip

\medskip
To start with the local case, we recall that  $H$ is said to be
\begin{itemize}
\item {\it primary} if $\mathfrak m \ne \emptyset$ and for all $a, b \in \mathfrak m$ there is an $n \in \N$ such that $b^n \subset aH$.

\item  {\it strongly primary} if $\mathfrak m \ne \emptyset$ and for every $a \in \mathfrak m$ there is an $n \in \N$ such that $\mathfrak m^{n} \subset aH$. We denote by $\mathcal M (a)$ the smallest $n$ having this property. 

\item {\it finitely primary} if there exist $s,\alpha\in \N$ such that $H$ is a submonoid of a factorial monoid $F=F^{\times}\times [q_1,\ldots,q_s]$ with $s$ pairwise non-associated primes $q_1,\ldots, q_s$ satisfying
$$H\setminus H^{\times}\subset q_1\cdot\ldots\cdot q_sF \text{ and } (q_1\cdot\ldots\cdot q_s)^{\alpha}F\subset H\,.$$

\item a {\it discrete valuation monoid} if it is primary and contains a prime element (equivalently, $H_{\red} \cong (\N_0,+)$).
\end{itemize}

Every strongly primary monoid is a primary \BF-monoid (\cite[Section 2.7]{Ge-HK06a}). If $H$ is finitely primary, then $H$ is strongly primary, $F=\widehat{H}$, $|\mathfrak X(H)| = s$ is called the rank of $H$, and $H$ is seminormal if and
only if
$$H=H^{\times}\cup q_1\cdot\ldots\cdot q_s F\,.$$

\begin{theorem}\label{th3}
Let $H$ be a strongly primary monoid that is not half-factorial.
\begin{enumerate}
\item $H$ is asymptotic fully elastic. More precisely, $\overline{R}(H)=\{\rho(H)\}$.

\item $\rho(H)$ is the only limit point of the set $\{\rho(L)\mid L\in \mathcal L(H)\}$.

\item If $H$ is a seminormal finitely primary monoid with rank $\ge 2$, then $\{\rho(L)\mid L\in \mathcal L(H)\}=\{\frac{n}{2}\mid n\in \N_{\ge 2}\}$.
\end{enumerate}
\end{theorem}

\begin{proof}
We start with the following claim.

\medskip
\noindent{\bf Claim A. } Let $(a_k)_{k=1}^{\infty}$ be a sequence with  $a_k\in H$ such that $\lim_{k\rightarrow \infty}\max\mathsf L(a_k)=\infty$. Then $\liminf_{k\rightarrow \infty}\rho(a_k)\ge \rho(H)$.

\medskip
\noindent{\it Proof of \bf Claim A. } Let $b\in H\setminus H^{\times}$. By definition, for each $k\in \N$, there exists $n_k\in \N$ such that $b^{n_k}\t a_k$ and $b^{n_k+1}\nmid a_k$. Suppose $a_k=b^{n_k}b_k$ for each $k\in \N$, where $b_k\in H$. Since $H$ is strongly primary, we infer that $\max\mathsf L(b_k)<\mathcal M(a)$. It follows by $\lim_{k\rightarrow \infty}\max\mathsf L(a_k)=\infty$ that $\lim_{k\rightarrow \infty}n_k=\infty$. Therefore
$$\rho(a_k)\ge \frac{n_k\max\mathsf L(b)+\max\mathsf L(b_k)}{n_k\min\mathsf L(b)+\min\mathsf L(b_k)}\ge \frac{n_k\max\mathsf L(b)+\mathcal M(a)}{n_k\min\mathsf L(b)+\mathcal M(a)}$$
and hence $\liminf_{k\rightarrow \infty} \rho(a_k)\ge \rho(b)$. The assertion follows immediately. \qed[End of {\bf Claim A.}]

\medskip
1. We only need to prove $\overline{R}(H)\subset \{\rho(H)\}$. Let $a\in H\setminus H^{\times}$. Then $\lim_{k\rightarrow\infty}\max\mathsf L(a^k)=\infty$. By {\bf Claim A.}, $\overline{\rho}(a)=\lim_{k\rightarrow \infty}\rho(a^k)\ge \rho(H)$. It follows by $\sup \overline{R}(H)\le \rho(H)$ that  $\overline{R}(H)= \{\rho(H)\}$.

\medskip
2. Assume to the contrary that there is another limit point $r\in \R$ with $1\le r<\rho(H)$. Then there exists a sequence $(a_k)_{k=1}^{\infty}$ such that $\lim_{k\rightarrow \infty}\rho(a_k)=r$ and $\lim_{k\rightarrow\infty}\max\mathsf L(a_k)=\infty$.
It follows by {\bf Claim A. } that $r=\lim_{k\rightarrow \infty}\rho(a_k)\ge \rho(H)$, a contradiction to $r<\rho(H)$.

\medskip
3. Let $H$ be a seminormal finitely primary monoid with rank $r\ge 2$. Then $H=H^{\times}\cup q_1\cdot\ldots\cdot q_s\widehat{H}$, where $q_1,\ldots,q_s$ are $s$ pairwise non-associated primes of $\widehat{H}$. 

Let $a\in H$. If $a$ is an atom or a unit, then $\rho(a)=1=\frac{2}{2}$. Otherwise $a=\epsilon q_1^{k_1}\cdot\ldots\cdot q_s^{k_s}$ with $\epsilon\in \widehat{H}^{\times}$ and $k_1,\ldots, k_s\in \N_{\ge 2}$. Since $q_1q_2^{k_2-1}\cdot\ldots\cdot q_s^{k_s-1}$ and $\epsilon q_1^{k_1-1}q_2\cdot\ldots\cdot q_s$ are atoms, it follows that $2\in \mathsf L(a)$ and $\rho(a)=\frac{\max \mathsf L(a)}{2}$. Thus $\{\rho(L)\mid L\in \mathcal L(H)\}\subset\{\frac{n}{2}\mid n\in \N_{\ge 2}\}$.

Let $n\in \N_{\ge 2}$ and $a=(q_1\cdot\ldots\cdot q_s)^n$. Then $\min \mathsf L(a)=2$ and $\max\mathsf L(a)=n$ which imply that $\rho(a)=\frac{n}{2}$. Therefore $\{\rho(L)\mid L\in \mathcal L(H)\}=\{\frac{n}{2}\mid n\in \N_{\ge 2}\}$.
\end{proof}

\bigskip
We say $H$ is 
 {\it weakly Krull} (\cite[Corollary 22.5]{HK98}) if
$ 
H = \bigcap_{{\mathfrak p} \in \mathfrak X (H)} H_{\mathfrak p}$  and $  \{{\mathfrak p} \in \mathfrak X (H) \mid a \in {\mathfrak p}\}$ is finite for all $a \in H$  and $H$ is {\it weakly factorial} if  one of the following equivalent conditions is satisfied (\cite[Exercise 22.5]{HK98}){\rm \,:}
     \begin{itemize}
     \item Every non-unit is a finite product of primary elements.

     \item $H$ is a weakly Krull monoid with trivial $t$-class group.
     \end{itemize}

Clearly, every localization $H_{\mathfrak p}$ of  $H$ at a minimal prime ideal $\mathfrak p \in \mathfrak X (H)$ is primary, and a weakly Krull monoid $H$ is $v$-noetherian if and only if  $H_{\mathfrak p}$ is $v$-noetherian for each $\mathfrak p \in \mathfrak X (H)$. Every $v$-noetherian primary monoid is strongly primary and $v$-local (\cite[Lemma 3.1]{Ge-Ha-Le07}), and
every strongly primary monoid is a primary \BF-monoid (\cite[Section 2.7]{Ge-HK06a}). Therefore the coproduct of a family of strongly primary monoids is a \BF-monoid,
and every coproduct of a family of primary monoids is weakly factorial. A $v$-noetherian weakly Krull monoid $H$ is weakly factorial if and only if $\mathcal C_v (H)=0$ if and only if $H_{\red} \cong \mathcal I_v^* (H)$.

Let $R$ be a domain. We say $R$ is a Mori domain, if $R^{\bullet}$ is $v$-noetherian. $R$ is weakly Krull (resp. weakly factorial) if and only if its multiplicative monoid $R^{\bullet}$  is weakly Krull (resp. weakly factorial).
Weakly Krull domains were introduced by D. D. Anderson, D. F. Anderson, Mott, and Zafrullah (\cite{An-An-Za92b, An-Mo-Za92}). We recall some most basic facts.
The monoid $R^{\bullet}$ is primary if and only if $R$ is one-dimensional and local. If $R$ is a one-dimensional local Mori domain, then $R^{\bullet}$ is strongly primary(\cite[Proposition 2.10.7]{Ge-HK06a}). Furthermore, every one-dimensional semilocal Mori domain with nontrivial conductor is weakly factorial and the same holds true for generalized Cohen-Kaplansky domains. It can be seen from the definition that  one-dimensional noetherian domains are $v$-noetherian weakly Krull domains.

We continue with $T$-block monoids which are weakly Krull monoids of a combinatorial flavor and  are used to model general weakly Krull monoids.
Let $G$ be an additive abelian group, $G_0 \subset G$ a
subset, $T$ a reduced monoid and $\iota \colon T \to G$ a homomorphism. Let
$\sigma \colon \mathcal F(G_0) \to G$ be the unique homomorphism
satisfying $\sigma (g) = g$ for all $g \in G_0$. Then
\[
B = \mathcal B (G_0,T,\iota) = \{S\,t \in \mathcal F(G_0) \time T\,\mid\, \sigma (S) + \iota(t) =
0\,\} \subset \mathcal F (G_0) \time T = F
\]
the  \textit{$T$-block monoid over  $G_0$  defined by $\iota$}\,. For details about $T$-block monoids, see \cite[Section 4]{Ge-Ka-Re15a}.

 Let $D$ be another monoid. A homomorphism $\phi: H\rightarrow D$ is said to be
 \begin{itemize}
 \item divisor homomorphism  if $\phi(u)\t \phi(v)$ implies that $u\t v$ for all $u,v\in H$.
 
 \item cofinal if for every $a\in D$, there exists $u\in H$ such that $a\t \phi(u)$.
 \end{itemize}
 Suppose $\phi: H\rightarrow D$ is a divisor homomorphism and $D$ is reduced.
 Then $\phi(H)$ is a saturated submonoid of $D$ and $H_{\red}$ is isomorphic to $\phi(H)$(see \cite[Proposition 2.4.2.4]{Ge-HK06a}). Therefore we can view $H_{\red}$ as a saturated submonoid of $D$.

\medskip
The following is the main theorem of this section. Clearly, orders $R$ in algebraic number fields satisfy all assumptions of Theorem \ref{1.1}. In particular, $\mathcal C_v(R)$ is finite, it coincides with the Picard group, and every class contains a regular prime ideal (whence $G_{\mathcal P}=\mathcal C_v(H)$ holds). Suppose $R$ is a $v$-noetherian weakly Krull domain with nonzero conductor. Then all localizations $R_{\mathfrak p}$ are finitely primary and if $R$ is not semilocal, then $R$ has a regular element with is not a unit whence $\mathcal B(G_{\mathcal P})\neq \{1\}$.
 For an extended list of examples, we refer to  \cite[Examples 5.7]{Ge-Ka-Re15a}. 

\begin{theorem} \label{1.1}
Let $H$ be a  $v$-noetherian weakly Krull monoid with the conductor $\emptyset \ne \mathfrak f = (H \DP \widehat H)\subsetneq H$ such  that the localization $H_{\mathfrak p}$ is finitely primary  for each minimal prime ideal $\mathfrak{p}\in\mathfrak{X}(H)$. We set $\mathcal P^*=\{\mathfrak p\in \mathfrak X(H)\mid \mathfrak p\supset \mathfrak f\}$ and suppose $\mathcal P=\mathfrak X(H)\setminus \mathcal P^*\neq\emptyset$. Let $G_{\mathcal P}\subset \mathcal C_v(H)$ denote the set of classes containing a minimal prime ideal from $\mathcal P$, and let $\pi: \mathfrak {X}(\widehat{H})\rightarrow \mathfrak{X}(H)$ be the natural map defined by $\pi(\mathfrak P)=\mathfrak P\cap H$ for all $\mathfrak P\in \mathfrak {X}(\widehat{H})$.
\begin{enumerate}
\item Suppose $\pi$ is bijective.  If  $\widehat{H_{\mathfrak p}}^{\times}/H_{\mathfrak p}^{\times}$ is finite for each minimal prime ideal $\mathfrak{p}\in\mathfrak{X}(H)$ and $\mathcal B(G_{\mathcal P})\neq \{1\}$, then $H$ is fully elastic.

\item Suppose $\pi$ is not bijective.  If  $H$ is seminormal and $G_{\mathcal P}=\mathcal C_v(H)$ is  finite, then $\rho(H)=\infty$ and $H$ is fully elastic.

\end{enumerate}

\end{theorem}

\begin{proof}

Let $\delta_H \colon H \to \mathcal I_v^* (H)$ be the homomorphism defined by $\delta_H(a)=aH$. Then $\delta_H$ is a cofinal divisor homomorphism with $\mathcal C(\delta_H)=\mathcal C_v(H)$ by \cite[Proposition 2.4.5]{Ge-HK06a}.

Since $H$ is $v$-noetherian, we obtain   $\mathcal P^*$ is finite and non-empty by \cite[Theorem 2.2.5.1]{Ge-HK06a}, say $\mathcal P^*  = \{\mathfrak p_1, \ldots, \mathfrak p_n\}$ with $n \in \N$.
By \cite[Theorem 2.6.5.3]{Ge-HK06a} that $H_{\mathfrak p}$ ia a discrete valuation monoid for each $\mathfrak p\in \mathcal P$ and
by \cite[Proposition 5.3.4]{Ge-Ka-Re15a}, there exists an isomorphism
\[
\chi \colon \mathcal I_v^*(H) \to D=\coprod_{\mathfrak p\in \mathfrak X(H)}(H_{\mathfrak p})_{\red} = \mathcal F(\mathcal P) \time (H_{\mathfrak p_1})_{\red} \time \ldots \time (H_{\mathfrak p_n})_{\red}
\]
where $\chi \t \mathcal P = \id_{\mathcal P}$ and, for all $i \in [1,n]$, $D_i := (H_{\mathfrak p_i})_{\red}$ is a reduced  finitely primary monoid.
Hence \,$\chi \circ \delta_H \colon H  \to D$ is a cofinal divisor homomorphism with $\mathcal C(\chi \circ \delta_H)\cong \mathcal C_v(H)$ and we can view $H_{\red}$ as a saturated submomoid of $D$.

By \cite[ Proposition 3.4.8]{Ge-HK06a}, it is sufficient to prove the assertions for the associated $T$-block monoid
\[
B = \mathcal B (G_{\mathcal P}, T, \iota) \subset F = \mathcal F (G_{\mathcal P}) \time T \,,
\]
where $T = D_1 \times \ldots \times D_n$ and $\iota \colon T \to G$ is defined by $\iota (t) = [t]\in \mathcal C_v(H)$ for all $t \in T$.

1. Since  $\pi$ is bijective, it follows by \cite[Lemma 5.1.3]{Ge-Ka-Re15a}  that $H_{\mathfrak p}$ has rank $1$ for each $\mathfrak p\in \mathfrak X(H)$.
Since $\widehat{H_{\mathfrak p_i}}^{\time}/H_{\mathfrak p_i}^{\time}$ is finite for each $i\in [1,n]$, we have $D_i$ is finitely generated for each $i\in [1,n]$. It follows that $F$ is locally finitely generated. Therefore  $B$ is locally finitely generated.

Since $\mathcal B(G_{\mathcal P})\neq \{1\}$ implies that  there exists a half-factorial subset $G_1\subset G_P$ such that $\mathcal B(G_1)\neq \{1\}$ by \cite[Lemma 5.4]{Ge-Sc-Zh17b}, it follows that
$\min \overline{R}(\mathcal B(G_{\mathcal P}))=1$. Hence  $\min \overline{R}(B)=1$ and $B$ is fully elastic by Theorem \ref{1}.

\medskip
2.  Since $\pi$ is not bijective, it follows by \cite[Lemma 5.1.3]{Ge-Ka-Re15a}  that  there exists $i\in [1,n]$ such that the rank $s_i$ of  $H_{\mathfrak p_i}$ is larger than $1$, say i=1. Then $D_1$ is a seminormal reduced finitely primary monoid of rank $s\ge 2$. Suppose $\widehat{D_1}=\widehat{D_1}^{\times}[p_1,\ldots,p_s]$, where $p_1,\ldots,p_s$ are $s$ pairwise non-associated prime elements. Then $D_1=p_1\cdot\ldots\cdot p_s \widehat{D}\cup \{1\}$. Let $a=\epsilon p_1^{\alpha_1}\cdot\ldots\cdot p_s^{\alpha_s}\in D_1$ with $\epsilon\in \widehat{D_1}^{\times}$ and  $a\in \mathcal A(B)$ such that $|a|=\alpha_1+\ldots+\alpha_s$ is minimal and let  $n=|G_{\mathcal P}|$.

If there exists $\delta\in \widehat{D_1}^{\times}$ such that $b=\delta p_1\cdot\ldots\cdot p_s\in B$, then $b\in \mathcal A(B)$.  We infer that $c_k=\delta p_1p_2^{kn+1}\cdot\ldots\cdot p_s^{kn+1}$ and $d_k=\delta^{kn+1} p_1^{kn+1}p_2\cdot\ldots\cdot p_s$ are atoms of $B$ for all $k\in \N$.
Since $\max\mathsf L(c_kd_k)=kn+2$, we obtain that $\rho(c_kd_k)=\frac{kn+2}{2}$ for all $k\in \N$ and $\max\mathsf L(c_kd_k)- \min\mathsf L(c_kd_k)=kn$.

Otherwise let $\iota(p_1\cdot\ldots\cdot p_s)=g\neq 0$. Then $p_1\cdot\ldots\cdot p_s (-g)$ is an atom of $B$.
 We infer that $c_k=p_1p_2^{k\alpha_2 n+1}\cdot\ldots\cdot p_s^{k\alpha_s n+1}(-g)$ and $d_k= p_1^{k\alpha_1 n+1}p_2\cdot\ldots\cdot p_s(-g)$ are atoms of $B$ for all $k\in \N$.
Since $\max\mathsf L(c_kd_k)=kn+2$, we obtain that $\rho(c_kd_k)=\frac{kn+2}{2}$ for all $k\in \N$ and $\max\mathsf L(c_kd_k)- \min\mathsf L(c_kd_k)=kn$.

Thus in both cases there is a sequence $(a_k)_{k=1}^{\infty}$ with term $a_k\in H$ such that $\rho(a_k)=(kn+2)/2$ and $\max\mathsf L(a_k)-\min\mathsf L(a_k)=kn$. Therefore $\rho(H)=\infty$.

Since the sequence $0_{G_{\mathcal P}}\in\mathcal F(G_{\mathcal P})$ is a prime element of  $B$,  it follows by \cite[Lemma 2.11]{B-C-C-K-W06}  that
\begin{align*}
&\bigcup _{k=1}^{\infty}\{a/b\in \Q\mid 1\le a/b\le \rho(a_k), a-b\t \max\mathsf L(a_k)- \min\mathsf L(a_k)\}\\
=&\bigcup _{k=1}^{\infty}\{a/b\in \Q\mid 1\le a/b\le \frac{kn+2}{2}, a-b\t kn\}\\
\subset& \{\rho(L)\mid L\in \mathcal L(H)\}\,.
\end{align*}

Let $q=a/b\in \Q$ with $q\ge 1$, where $a,b\in \N$ with $\gcd(a,b)=1$.
Choose $k_0=2a(a-b)$. Then $a/b\le k_0/2\le(k_0n+2)/2$ and $a-b\t k_0n$.
It follows that  $$q\in \{a/b\in \Q\mid 1\le a/b\le \frac{k_0n+2}{2}, a-b\t k_0n\}\subset \{\rho(L)\mid L\in \mathcal L(H)\}\,.$$
Therefore $H$ is fully elastic.
\end{proof}


\providecommand{\bysame}{\leavevmode\hbox to3em{\hrulefill}\thinspace}
\providecommand{\MR}{\relax\ifhmode\unskip\space\fi MR }
\providecommand{\MRhref}[2]{%
  \href{http://www.ams.org/mathscinet-getitem?mr=#1}{#2}
}
\providecommand{\href}[2]{#2}

\end{document}